\documentclass[12pt]{amsart}
\usepackage{latexsym,fancyhdr,amssymb,color,amsmath,amsthm,graphicx,listings,comment}
\usepackage[section]{placeins}
\newtheorem{thm}{Theorem} \newtheorem{lemma}{Lemma} \newtheorem{propo}{Proposition} \newtheorem{coro}{Corollary}
\definecolor{red1}{rgb}{1,0.9,0.9} \definecolor{blue1}{rgb}{0.9,0.9,1} \definecolor{green1}{rgb}{0.9,1,0.9} 
\definecolor{yellow1}{rgb}{1,1,0.9} \definecolor{yellow2}{rgb}{1,1,0.8}

\let\paragraph\subsection
\newcommand{\str}{\rm{str}}
\newcommand{\NN}{\mathbb{N}}

\title{One can hear the Euler characteristic of a simplicial complex}
\author{Oliver Knill}
\date{November 26, 2017}
\address{Department of Mathematics \\ Harvard University \\ Cambridge, MA, 02138 }

\begin{document}
\maketitle

\begin{abstract}
We prove that that the number $p$ of positive eigenvalues of the 
connection Laplacian $L$ of a finite abstract simplicial complex $G$ matches the number
$b$ of even dimensional simplices in $G$ and that the number $n$ of negative eigenvalues
matches the number $f$ of odd-dimensional simplices in $G$. The Euler characteristic
$\chi(G)$ of $G$ therefore can be spectrally described as $\chi(G)=p-n$. This is in contrast
to the more classical Hodge Laplacian $H$ which acts on the same Hilbert space, 
where $\chi(G)$ is not yet known to be accessible from the spectrum of $H$.
Given an ordering of $G$ coming from a build-up as a CW complex, every simplex
$x \in G$ is now associated to a unique eigenvector of $L$ and the correspondence
is computable. The Euler characteristic is now not only the
potential energy $\sum_{x \in G} \sum_{y \in G} g(x,y)$
with $g=L^{-1}$ but also agrees with a logarithmic energy
${\rm tr}(\log(i L)) 2/(i \pi)$ of the spectrum of $L$.
We also give here examples of isospectral but non-isomorphic 
abstract finite simplicial complexes. One example shows that we can not hear the 
cohomology of the complex. 
\end{abstract}

\section{Introduction}

\paragraph{}
The energy theorem \cite{Spheregeometry} identifies the Euler characteristic 
$\chi(G)$ of a finite abstract simplicial complex $G$ as a potential theoretical energy 
$$   E(G) = \sum_{x \in G} \sum_{y \in G} g(x,y)  \; , $$ 
where $g=L^{-1}$ is the Green's function, an integer matrix. Is there a spectral interpretation of 
Euler characteristic $\chi(G)=\sum_x (-1)^{{\rm dim}(x)}$ for a self-adjoint Laplacian of 
a simplicial complex? We answer this question with yes. We are not aware of any other self-adjoint 
Laplacian for a simplicial complex which determines $\chi(G)$ from its spectrum.  \\

\paragraph{}
There is a non-selfadjoint matrix which produces $\chi(G)$: 
if $\Gamma$ is an acyclic directed graph with $m$ vertices, define 
$A(i,j)=1$ if and only if $(i,j)$ is not an edge and $0$ else. 
Define the diagonal $D={\rm Diag}(x_1, \dots, x_m)$. In \cite{Stanley1996a}, 
the determinant formula ${\rm det}(1+D A) = \sum_{x} x_{1} x_{2} ,\dots x_{|x|}$ 
is proven, where the sum is over all paths $x$ in the graph and where $|x|$ denotes the number of vertices
in $x$. Applying this to the graph where $G$ is the vertex set and $(x,y)$ is 
and edge if $x \subset y$, and $x_j=-1$, then $A(i,j)=-1$ if and only if $i$ is not a subset of $j$. 
and $0$ else, Bera and Muherjee note that this implies ${\rm det}(A) = (1-\chi(G))$ \cite{BeraKMukherjee}. 
As ${\rm det}(A)$ is the product of the
eigenvalues of a matrix $A$ defined by $G$, this is a spectral description. But the matrix $A$ is
not symmetric and in general has complex eigenvalues so that it can not be conjugated to a symmetric matrix.  \\

\paragraph{}
Besides pointing out that the spectrum of the connection Laplacian $L$ determines the Euler characteristic 
we also show that there is a canonical bijection between simplices of $G$ ordered as a CW complex
and the eigenvalues of $L$. Every simplex has so a well defined energy value and explicitly defined 
eigenvector of $L$ in the Hilbert space.
The simplices, the basic building blocks of a simplicial complex therefore have a particle feature.
They manifest as eigenstates of the energy matrix $g=L^{-1}$. Our result extends to the 
strong ring generated by simplicial complexes. For these more general objects, still the number of positive
minus the number of negative eigenvalues is the Euler characteristic. 
We already knew that the total energy $\chi(G)$ is a ring 
homomorphism from the ring to the integers and that the eigenvectors of the connection Laplacian are
multiplicative. Now, an eigenvalue to $\lambda \mu$ of $G \times H$ can 
be associated to a specific product simplex $x \times y$ with $x \in G, y \in H$. 

\paragraph{}
If $L$ is the connection Laplacian of a finite abstract simplicial complex $G$
defined by $L(x,y)={\rm sign}(|x \cap y|)$ then by definition,
the {\bf super trace} $\sum_{x \in G} \omega(x) L(x,x)$ of $L$ is the Euler 
characteristic $\chi(G) = \sum_{x \in G} \omega(x) = b-f$, where $\omega(x) = (-1)^{{\rm dim}(x)}
= (-1)^{|x|-1}$, where $b=|\{ x \in G \; | \; |x| \in 2\NN-1 \}|$
is the number of even-dimensional simplices and 
$f=|\{ x \in G \; | \; |x| \in 2\NN \}|$ is the number 
odd-dimensional simplices in $G$. The trace of $L$ is $b+f=|G|$ is a spectral invariant. 
The super trace however is not unitary invariant: it does not provide a spectral description 
of $\chi(G)$. We can not see $\chi(G)$ from the eigenvalues of the
Hodge Laplacian $H(G)=(d+d^*)^2$ yet. We know that $H$ decomposes into blocks $H_k(G)$ which 
have as nullity the Betti numbers $b_k(G)$. By Euler-Poincar\'e, one has then
$\chi(G) = \sum_{k=0}^{\infty} (-1)^k b_k(G)$. 
If we look at the spectrum of $H$ as a set of real numbers
without any labels, we can not see whether an eigenvalue comes from odd or even forms. 

\paragraph{}
The Dirac operator $D=(d+d^*)$ defines a super-symmetry between the nonzero 
eigenvalues of $H$ restricted to the even and odd dimensional part of the 
exterior bundle. The symmetry is $f \to D f$. This implies the celebrated 
McKean-Singer formula 
$\chi(G) = \str(\exp(-t H))$ which followed from the fact that $\str(H^k)=0$
for all $k>0$ and $\str(1)=\chi(G)$ \cite{McKeanSinger,knillmckeansinger}. 
Still, also here, as the super trace refers to the $\omega(x)$, 
the Euler characteristic is not yet a true spectral
property. Applying a unitary base change which exchanges some odd and
even dimensional parts can change the super trace. There is an 
analogue McKean-Singer formula 
$$  \chi(G) = \str(L^{-1}) = \str(L^1) = \str(L^0)  $$ 
for $L$. But the lack of super symmetry for $L$ does not allow us to find identities 
$\chi(G) = \str(L^{-t})$ for $t$ different from $0,1$ or $-1$. 

\section{A spectral formula}

\paragraph{}
We formulate now a spectral interpretation of the Euler characteristic $\chi=\chi(G)$.
Let $p=p(G)$ be the number of positive eigenvalues of the connection Laplacian $L$
and let $n=n(G)$ be the number of negative eigenvalues of $L$. Let $b(G)$ be the number
of even-dimensional simplices in $G$ and $f(G)$ denote the number of odd-dimensional 
simplices in $G$.

\begin{thm}[Spectral formula for Euler characteristic] 
For any complex $G$ we have $p(G)=b(G)$ and $n(G)=f(G)$ so that $\chi=p-n$.  
\end{thm}

\paragraph{}
As the proof will show, the result holds more generally for discrete CW complexes, where cells
are inductively added to spheres. It also holds for products of simplicial complexes which 
are no simplicial complexes any more. This can appear strange at first as we are familiar with 
the Hodge situation, which features a complete symmetry between positive and negative energy states.
While for the Hodge operator $H$ which lives on the same Hilbert space, the harmonic forms
are topologically interesting, it appears that the other eigenvalues and eigenvectors
change under coordinate changes. For the connection operator $L$ -  which has no kernel -
all eigenvectors appear to be of topological interest.

\paragraph{}
We have seen that the Hodge Laplacian $H$ is ``spectrally additive" in the sense that
the eigenvalues of $H(A \times B)$ are the sum of the eigenvalues of $H(A)$ 
and $H(B)$ (as for the Laplacian in the continuum) and that $L$ is ``spectrally multiplicative"
in that the eigenvalues of $L(A \times B)$ are the product of the eigenvalues of 
$L(A)$ and $L(B)$, which has no continuum analogue yet \cite{StrongRing}.
We see from this that the spectral formula goes over to the strong ring generated
by simplicial complexes. An eigenvector of $A \times B$ which corresponds to a 
positive energy state can either be decomposed into two positive energy eigenstates
or two negative eigenstates. 

\paragraph{}
{\bf Example 1:} Let $G=\{\{1\}$,$\{2\}$,$\{3\}$,$\{1,2\}$,$\{2,3\}\}$. Then 
$$ L=\left[ \begin{array}{ccccc} 1 & 0 & 0 & 1 & 0 \\ 0 & 1 & 0 & 1 & 1 \\
                  0 & 0 & 1 & 0 & 1 \\ 1 & 1 & 0 & 1 & 1 \\ 0 & 1 & 1 & 1 & 1 \\ \end{array} \right]  $$
which has the eigenvalues $(1 \pm \sqrt{5})/2$, $(3 \pm \sqrt{13})/2$ and $1$. 
There are $p=3$ positive eigenvalues and $n=2$ negative ones. 
This matches the Euler characteristic 
$\chi(G)=b-f=3-2$. 
%  s=CubeGraph; G=GraphComplex[s]; L=ConnectionLaplacianComplex[G]; Eigenvalues[L] 

{\bf Example 2:} The cube graph complex
$G=\{\{1\}$,$\{2\}$, $\{3\}$, $\{4\}$, $\{5\}$, $\{6\}$, $\{7\}$, $\{8\}$, $\{1,2\}$,
$\{1,3\}$, $\{1,4\}$, $\{2,5\}$, $\{2,6\}$, $\{3,5\}$, $\{3,7\}$, $\{4,6\}$, $\{4,7\}$, 
$\{5,8\}$, $\{6,8\}$, $\{7,8\}\}$ has a Laplacian $L$ 
with eigenvalues $\{ 3 \pm \sqrt{10}$, $(2 \pm \sqrt{5})^{(3)}$, $(1\pm \sqrt{2})^{(3)}$, $(-1)^{(5)}$, $1 \}$.
There are $p=8$ positive and $n=12$ negative eigenvalues. (There are $4$ negative eigenvalues $-1$ which do 
not match up naturally). There are $b=8$ zero dimensional simplices and $f=12$ one dimensional simplices 
so that $\chi(G)=8-12=-4$. 

{\bf Example 3:} For the cyclic complex $G=C_4=\{\{1\}$, $\{2\}$, $\{3\}$, $\{4\}$, $\{1,2\}$, $\{1,4\}$, $\{2,3\}$, $\{3,4\}\}$,
the eigenvalues are $\{ 2 \pm \sqrt{5}$, $(1 \pm \sqrt{2})^{(2)}$, $1$, $-1 \}$, 
featuring four positive and four negative eigenvalues, matching $\chi(G)=0$. 

\section{Proof}

\paragraph{}
We have seen that $g(x,x) = 1-\chi(S(x))$, where $S(x)$ is the unit sphere of $x$ 
in the graph $G_1=(V,E)$, where $V=G$ and where $(a,b) \in E$ if and only if 
$a \subset b$ or $b \subset a$. What are the other entries $g(x,y)$? 
An element $x$ in $G$ is called a {\bf facet}, if it is not contained in any other
simplex $y$ of $G$. 

\begin{lemma}
If $x$ is a facet and $y \subset x$, then $g(x,y)= \omega(y)$. 
\end{lemma}
\begin{proof}
We can rearrange the simplices so that the facet $x$ is the last one. 
The vector $\{ g(x,y) \; | \; y=1, \dots, n\}$ is then the last column of $g$. 
The Cramer formula tells that $g(x,y)$ is $\det(L(y))/\det(L)$, where $L(y)$
is the matrix $L$ in which column $x$ has been replaced by $e(y)$. 
We prove the statement by induction with respect to the number of elements $|G|$ 
of $G$. Now make a Laplace expansion with respect to row $x$. This gives a sum,
using induction as we now have a simplicial complex $G \setminus \{x\}$ and 
$\sum_{z \subset x, z \in S(y) \cup \{ y \} } w(z) = 1-\chi(S(y)) = \omega(y)$. 
\end{proof}

\paragraph{}
The proof of the theorem is done by induction on the number $|G|$ of elements in $G$. 
As in the proof of the unimodularity theorem \cite{Unimodularity}, we prove the result more generally for discrete
CW complexes, where we have a finer built-up and so a stronger induction assumption during the proof. 
Given a connection Laplacian $L$ to a complex $G$; if we add a new cell $x$ and consider 
the complex $H=G +_A x$ in which $x$ is attached to $A$. Let $L$ be the connection Laplacian 
of $G$ and $K$ the connection Laplacian of $H$. Define $K(t)$ as the matrix defined by 
$K(t)(y,z) = K(y,z)$ if $z$ is different from $x$ and $K(t)(y,x)= t K(y,x)$ if
$y \neq x$. 

\begin{propo}
${\rm det}(K(t)) = (1-t \cdot \chi(A)) {\rm det}(L)$. 
\end{propo}
\begin{proof}
Make a Laplace expansion with respect to the last column. The formula
$$ {\rm det}( \left[ \begin{array}{ccccccc}
 L_{11} & L_{12} & . & . & . & L_{1n} & t_1 K_{1,x}\\
 L_{21} & L_{22} & . & .  & . & . & t_2 K_{2,x} \\
 . & . & . & . & . & . & t_3 K_{3,x} \\
 . & . & . & . & . & . & .   \\
 . & . & . & . & . & . & .   \\
 . & . & . & . & . & . & .   \\
 L_{n1} & . & . & . & . & L_{nn} & t_n K_{n,x} \\
 1 & 1 & 1 & 1 & 1 & 1 & 1 \\
\end{array} \right]) = {\rm det}(L) - \sum_{y \subset x} t_y {\rm det}(L) \omega(y) \;  $$
immediately proves the statement. Just set all variables $t_k=t$. 
We have used the Lemma rephrased that the minor $\det(L[y,x])$ 
is $\omega(y)$. 
\end{proof}

\paragraph{}
We see that during the deformation $K(0)$ to $K(1)$, when adding an even-dimensional 
cell, all eigenvalues stay on the same side of $0$. When adding an odd dimensional
cell, the new eigenvalue $1$ which is present in $K(0)$ but not present in $L$
crossed from $t=0$ to $t=1$ from the positive to the negative side and no other eigenvalue 
crosses, as the characteristic polynomial has a simple root there. 
The formula also shows that if $\chi(A)=0$, then $\det(K(t))$ is constant and that $\chi(K(t))$
changes linearly in all cases. 

\begin{coro}
a) If $\chi(A)=0$, then the matrix $K(t)$ is invertible for all $t$. The number of positive
eigenvalues of $K$ has increased by $1$ and the number of negative eigenvalues has remained the same. \\
b) If $\chi(A)=2$, then the matrix $K(t)$ is not invertible exactly for $t=1/2$ and has
a one-dimensional kernel. The number of positive eigenvalues of $K$ and $L$ are the same. 
The number of negative eigenvalues has increased by $1$. 
\end{coro}

\paragraph{}
The proof of the result does not directly follow from the {\bf unimodularity theorem}
$|{\rm det}(L)|=1$ or the {\bf energy theorem} $\sum_{x,y} g(x,y) = \chi(G)$, if 
$g=L^{-1}$. We have shown the unimodularity theorem for
discrete CW complexes, which generalize simplicial complexes. Adding a cell $x$ along a sphere $S(x)$ 
multiplies the determinant by $1-\chi(S(x)) = (-1)^{{\rm dim}(x)}$. 
So, if we add an odd dimensional simplex, the determinant gets multiplied by $-1$
meaning that an other eigenvalue $-1$ is added. This does not prove the theorem yet
as it could be possible for example that the addition renders two negative eigenvalues
positive or even more complicated changes of the spectrum happen. 

\paragraph{}
We still have not completely cleared up the meaning of all the matrix entries $g(x,y)$. 
We know $g(x,x)=1-\chi(S(x))$ as this was essential for the energy theorem. Let $W^+(x,y) =W^+(x) \cap W^+(y)$ 
denote the intersections of the stable manifolds of $x$ and $y$. Here $W^+(x)$ is the set of simplices containing $x$
but not being equal to $x$. And $W^-(x)$ is the set of simplices contained in $x$ but not equal to $x$. 
We know therefore that $g(x,y) = (1-\chi(W^+(x,y)))(1-\chi(W^-(x,y)))$ holds if $x=y$ or $y \subset x$
and $x$ is a facet. In the later case this is $\omega(y)$. 
We also know that if $g(x,x)=0$ then $g(x,y)=0$ if $y \subset x$. We have tried in various
ways to modify the in general incorrect formula
$g(x,y) = (1-\chi(W^+(x,y)))(1-\chi(W^-(x,y))$ to get an identity which works for all 
$x,y$. 

\paragraph{}
The modification $g(x,y) = (\chi(x \cap y)-\chi(W^+(x,y)))( \chi(x \cap y) -\chi(W^-(x,y))$
comes already closer as it covers the cases we know, like when there is not connection at all between
$x$ and $y$. It still fails in general. More experimentation is needed. 
An other approach to the puzzle is to give meaning to the minors ${\rm det}(L[x,y])$
as $g(x,y) = (-1)^{i(x)+i(y)}{\rm det}(L[x,y])/{\rm det}(L)$, where $i(x)$ is the position in the
ordering of $G$ defining the matrix $L$. The determinant ${\rm det}(L[x,y])$ is a signed sum over all 
one dimensional paths where one of the connectivity components goes from $x$ to $y$. Representing
Green function values as path integrals is a familiar theme in mathematical physics. 

\section{Corollaries}

\paragraph{}
As the number of positive eigenvalues corresponds to the
number of even-dimensional complete sub-complexes and the number of negative
eigenvalues corresponds to the number of odd-dimensional complete sub-complexes,
we know that $L$ is never positive or negative definite if the dimension of the 
complex $G$ is positive: 

\begin{coro}[There is anti-matter] For any CW complex of positive dimension,
the connection Laplacian is never positive nor negative definite. 
\end{coro}
\begin{proof}
A positive dimensional simplicial complex contains both even and odd dimensional simplices. 
\end{proof}

\paragraph{}
We can write the sign function ${\rm sign}(x)$ as the imaginary part of 
$\log(i x) (2/(i \pi))$. From the unimodularity theorem follows that
$$ \sum_k \log(i \lambda_k) = i (\frac{\pi}{2})  \chi(G) $$
if the branch of the log has been chosen so that $\log(z) \in [-\pi/2,\pi/2]$ for 
${\rm Re}(z) \geq 0$ or equivalently 

\begin{coro}[Potential theoretical energy]
$$  \chi(G) = \frac{2}{i \pi} {\rm tr}(\log(i L)) \; . $$
\end{coro}

\section{Isospectral complexes}

\paragraph{}
The question whether the spectrum of $L(G)$ determines $G$
is equivalent to the question whether the adjacency spectrum of the 
connection matrix determines $G$. The topic of finding relations between 
a geometry space and the spectrum of its Laplacian is a classical 
one \cite{Kac66} and has been ported to finite geometries like graphs. 
We give in this section examples of non-isospectral complexes with respect to the 
connection Laplacian. 

\paragraph{}
As the eigenvalues of $L$ do not determine the complex $G$ one can ask whether the
eigenvectors do. We know that the eigenvalues and eigenvectors together determine
$L$ and so $G$. Now we can look for the class of unitary matrices $U$ for which
the operator $U D U^T$ has an inverse $g$ for which the sum $\sum_{x,y} g(x,y)$ is
the Euler characteristic. 

\begin{figure}[!htpb]
\scalebox{0.24}{\includegraphics{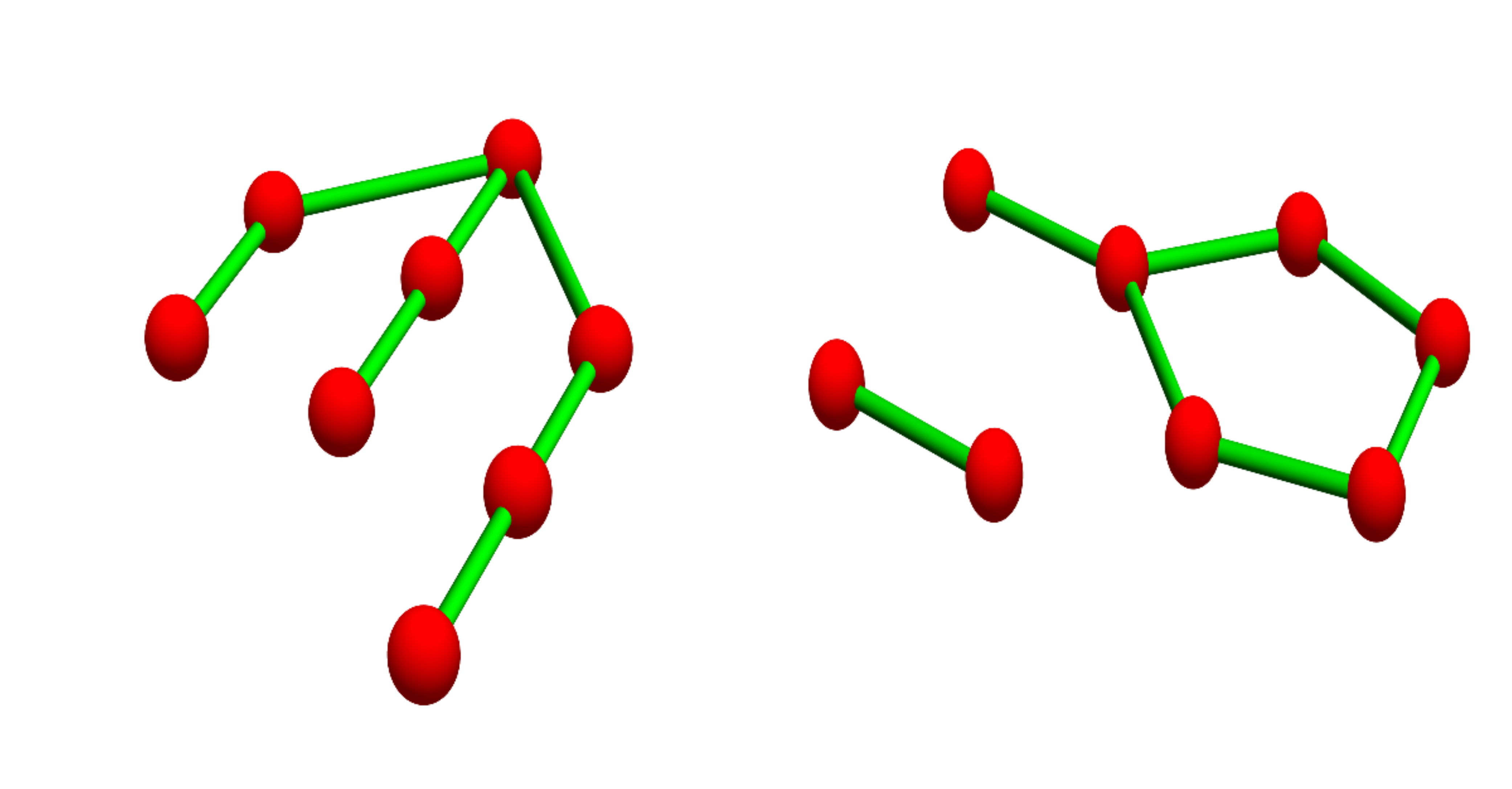}}
\label{example}
\caption{
{\bf Example 1:} The complex
$G=\{\{1\}$, $\{2\}$, $\{3\},\{4\}$, $\{5\},\{6\}$, $\{7\},\{8\}$, $\{1,2\},\{1,3\}$, 
$\{2,6\},\{2,7\}$, $\{4,5\}$, $\{6,8\},\{7,4\}\}$ and the complex
$H=\{\{1\},\{2\}$, $\{3\},\{4\}$, $\{5\},\{6\}$, $\{7\},\{8\}$, $\{1,2\},\{1,5\}$, 
$\{1,7\},\{2,8\},\{3,4\},\{5,6\},\{8,6\}\}$ are isospectral. 
}
\end{figure}

\begin{figure}[!htpb]
\scalebox{0.24}{\includegraphics{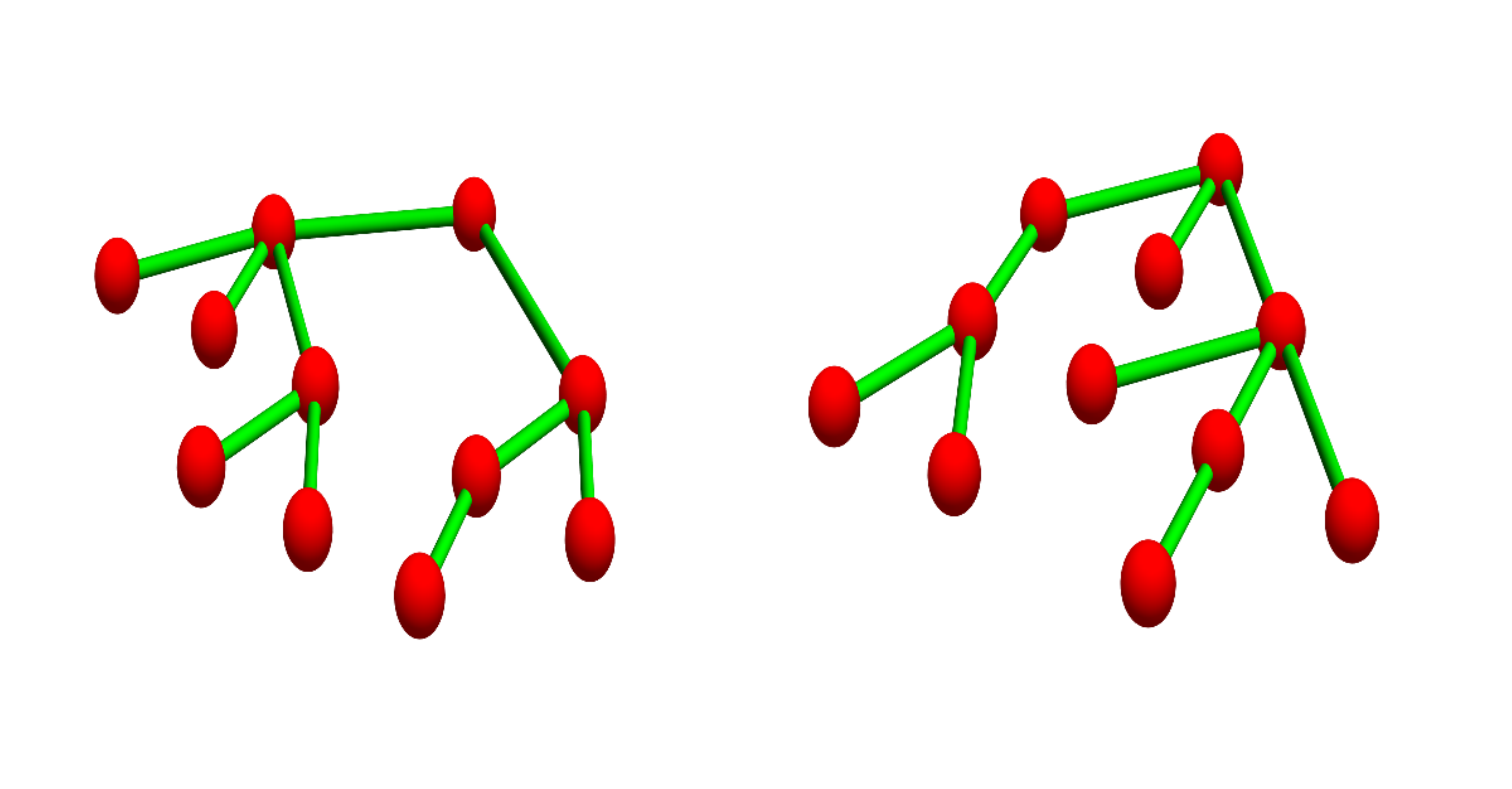}}
\label{example}
\caption{
{\bf Example 2:} The complex
$G=\{\{1\},\{2\},\{3\}$, $\{4\},\{5\},\{6\},\{7\}$, $\{8\}$, $\{9\},\{10\},\{11\}$, $\{1,2\},\{2,3\}$,
   $\{2,4\},\{2,5\},\{4,6\},\{5,8\},\{5,11\}$, $\{6,7\},\{6,10\},\{7,9\}\}$ and
the complex 
$H=\{\{1\},\{2\},\{3\},\{4\},\{5\},\{6\},\{7\}$, $\{8\},\{9\},\{10\}$, $\{11\},\{1,2\},\{1,3\}$,
   $\{2,5\},\{2,6\},\{3,4\}$, $\{3,7\},\{7,8\}$, $\{7,9\},\{7,11\},\{9,10\}\}$ 
are isospectral. 
}
\end{figure}

\paragraph{}
Remarkably, the {\bf Wu characteristic} 
$$  \omega(G) =  \omega_2(G) = \sum_{x,y, x \cap y \neq \emptyset} \omega(x) \omega(y) $$ 
of the two complexes in these examples agree. In example 1), the Wu characteristic is $\omega_2=1$.
Also the higher Wu characteristic agree: $\omega_3=-5$ and $\omega_4=13$. 
In example 2), there are both the same $9$. Also the higher Wu characteristic
$\omega_3(G)=-35, \omega_4(G)=105$ both do agree. This prompts the question: \\

{\bf Question:} are the Wu characteristics $\omega_k$ determined by the spectrum?
We remind that  $\omega_k(G)$ is the sum over all products $\omega(x_1) \cdots \omega(x_k)$ 
of simultaneously intersecting simplices $x_1, \dots, x_k$ in $G$. So far we only know the answer
for the Euler characteristic $\omega_1(G)= \chi(G)$. 

\bibliographystyle{plain}

\end{document}